\documentclass[10pt]{amsart}
\usepackage{amsmath, amssymb, amsthm,latexsym}
\usepackage{graphicx}
\newcommand\C{{\mathbb C}}
\newcommand\oC{\widehat {\mathbb C}}

\newcommand\R{{\mathbb R}}

\renewcommand\:{\colon}
\newcommand\sub {\subseteq}
\newcommand\ra {\rightarrow}
\renewcommand\Im{\mathop{\mathrm{Im}}}
\renewcommand\Re{\mathop{\mathrm{Re}}}

\newcommand\Om{\Omega}
\newcommand\ga{\gamma}

\newcommand\no{\noindent}

\newtheorem{theorem}{Theorem}[section]

\newtheorem{proposition}[theorem]{Proposition}

\theoremstyle{definition}

\title{
Uniformization by square domains}%
\author{Mario Bonk}
\address{Mario Bonk\\Department of Mathematics\\
University of California\\ Los Angeles, CA 90095, USA.} \email{mbonk@math.ucla.edu}
\thanks{The author was partially supported by NSF grant  DMS-1162471}

\dedicatory{Dedicated to David Minda on the occasion of his retirement}

\date{May 9, 2016}

\begin{document}

\abstract{Let $\Om$ be a finitely connected subdomain of the Riemann sphere $\oC$ with $\infty\in \Om$. We find an extremal problem for conformal maps  on $\Om$ with suitable normalization at $\infty$ whose unique solution is a map onto a square domain, that is, a domain in $\oC$ whose complementary components are (possibly degenerate) squares with sides parallel to the real or the  imaginary  axis.}
\endabstract

\maketitle

\section{Introduction}\label{s:Intro}
 
In the following,  $\Om$ will always denote  a domain (that is, an open and connected set) in the Riemann sphere $\oC=\C\cup\{\infty\}$
with $\infty\in \Om$. We consider the family $\mathcal{F}$ of all conformal maps 
$f\:\Om\ra f(\Om)\sub \oC$ with the normalization 
\begin{equation} \label{norm} 
f(z) = z+ \frac{a_1}{z}+\dots
\end{equation} 
for $z$ near $\infty$. 
It is a classical and well-known fact \cite[Chapter 5]{Go} that the functional  
\begin{equation} \label{slit}
 f\in \mathcal{F}\mapsto \mathcal{L}(f):= \Re(a_1)
 \end{equation}  
has a unique minimum that is achieved by a conformal map $f=f_0\:\Om \ra D$ onto a parallel slit domain $D\sub \oC$
with slits  parallel to the imaginary  axis in $\C$.    Here by definition a {\em parallel slit domain} $D$ is a domain 
in $\oC$ with $\infty\in D$ whose complementary 
components are  closed line segments parallel 
to a fixed direction or single points (viewed as degenerate line segments). 
If one considers the more general functional 
 \begin{equation} \label{slit2}
 f\in \mathcal{F}\mapsto \mathcal{L}_\alpha(f):=\Re(e^{i\alpha} a_1)
 \end{equation}  
 for fixed $\alpha\in \R$, then it has a unique minimum given by a conformal map of $\Om$ onto a parallel  slit domain with slits parallel to the line $\ell_\alpha=\{ i e^{i\alpha/2}t:t\in \R\}$.  
 
 Note that finding the maximum 
 of $\mathcal{L}_\alpha$ is equivalent to finding the minimum of $\mathcal{L}_{\alpha+\pi}$. 
 In particular, the functional $\mathcal{L}=\mathcal{L}_0$ has a unique maximum given by a 
 conformal map onto a parallel  slit domain
with  slits  parallel to the real   axis.
 
These facts are behind a conceptually simple  proof of the statement that  every domain $\Om$ in $\C$ can be uniformized by a parallel slit domain $D$ (meaning that there exists a conformal map $f\:\Om \ra D$):
 one first shows by a compactness argument that a minimum $f=f_0$ of the functional in \eqref{slit} exists and 
  then  that $f_0$ maps onto a slit domain by a variational argument. 
  
 In contrast to this situation, it is an open problem
 whether each domain $\Om$ can be uniformized 
 by a {\em circle domain}, i.e.,  a domain in $\oC$ 
 whose complementary components are single points or Euclidean  disks. This is known as 
 Koebe's ``Kreisnormierungsproblem". 
 
 On the positive side, 
 if $\Om$ is finitely connected or, more generally,
 has countably many complementary components, then $\Om$ can be uniformized by a circle domain
 (see \cite[Chapter 5]{Go} for the classical case of finitely connected domains and \cite{HS} for domains with countably many complementary components). The proof of this  fact is much more involved than the simple argument leading to uniformization by parallel  slit domains as outlined  above. No simple extremal problem is known that would solve the ``Kreisnormierungspro\-blem", even in the finitely connected case.
 
   The main point of the present note is the observation that there is such an extremal problem 
  for the uniformization of finitely connected domains onto square domains. By definition a {\em square
  domain} $D$ is a domain in $\oC$ with $\infty
  \in D$ whose complementary components are single points or closed squares with sides parallel to the real or the imaginary axis.  
    
  To formulate this extremal problem, we fix some notation. If $f\in \mathcal{F}$ is a conformal map defined on a finitely connected  $\Om$ as above, we denote by $D=f(\Om)$ its 
  image domain and by $K_j$, $j=1, \dots, n$, the complementary components of $D$ in $\oC$. We allow 
 $n=0$ here. Then  $D=\Om=\oC$ and the class $\mathcal{F}$ just consists of the identity map.    We could  ignore components of $\Om$ and $D$ that are points, because an isolated point in the complement of $\Om$ is a removable singularity. For simplicity, we allow such components. 
  
  Each  set $K_j$  is a compact subset of $\C$.
 We denote by $A_j$ the Euclidean area (that is, the Lebesgue measure)   of  $K_j$ and by 
 $V_j$ its {\em vertical variation} defined as 
 $$ V_j= \max_{w\in K_j} \Im(w)-\min_{w\in K_j}\Im(w).$$ 
 Of  course, the quantities $A_j$ and $V_j$ depend on $f$, but for simplicity we suppress this dependence in our notation. 
  
We can now state our main theorem.

 \begin{theorem}\label{main}
 Let $\Om$ be a finitely connected domain in $\oC$ 
 with $\infty\in \Om$. Among all conformal maps 
 $f\: \Om \ra D=f(\Om)\sub \oC$ with the normalization
 \begin{equation} \label{norm2} 
 f(z) = z+ \frac{a_1}{z}+\dots 
 \end{equation} 
 the functional 
 \begin{equation}\label{funct} 
 f\in \mathcal{F} \mapsto \mathcal{S}(f) :=2\pi \Re(a_1)+
 \sum_{j=1}^n (V_j^2-A_j)
 \end{equation} 
 has a unique minimizer $f=f_0\in \mathcal{F}$. 
 The map $f_0$ is the unique conformal 
 $ f_0\: \Om \ra D$ map with the normalization \eqref{norm2} onto a square domain $D$.  \end{theorem}

\section{Proof of Theorem~\ref{main}} \label{s:proof}

Let $\Omega$ be as in Theorem~\ref{main}. It is a known fact 
that there exists a conformal map $g$  of $\Om$ onto a square domain 
$\widetilde \Omega$ with the normalization 
$$g(z) = z+ \frac{b_1}{z}+\dots $$
near $\infty$  (this follows from the Brandt-Harrington Uniformization Theorem; for the formulation of this theorem and its  proof see \cite{S96}).    We denote by $\mathcal{\widetilde F}$ the class of conformal maps 
$f$ on $\widetilde \Om$ normalized as in \eqref{norm2},   and by $\widetilde {\mathcal{S}}$
the functional on $ \mathcal{\widetilde F}$ defined as in \eqref{funct}. 
The map $f\in \mathcal{F}\mapsto \widetilde f:= f\circ g^{-1}$ is a bijection between $ \mathcal{F}$   and $\mathcal{\widetilde F}$.
Moreover, if 
$$ f(z) = z + \frac{a_1}{z}+\dots \text { and }  
 \widetilde f(z) = z + \frac{\widetilde a_1}{z}+\dots $$ 
 near $\infty$, then 
 $\tilde a_1= a_1-b_1$. Since the maps $f$ and $\tilde f$ have the same image domain $D=f(\Omega)= \widetilde f(\widetilde \Omega)$, we have 
 $$ \widetilde {\mathcal{S}}(\widetilde f) = \mathcal{S}(f)- 2\pi \Re(b_1). $$  
 So under the bijection $f\leftrightarrow \widetilde f$, the functionals ${\mathcal{S}}$
and  $\widetilde {\mathcal{S}}$   correspond to each other and just differ by the fixed additive constant 
 $-2\pi \Re(b_1)$.   
 In this way, the proof of 
 Theorem~\ref{main} is reduced  to the case where $\Om$ is a square domain to begin with.  

The theorem will now follow from the following 
statement. 

\begin{proposition}\label{prop:main}  
 Let $\Om$ be a finitely connected square domain in $\oC$ 
 with $\infty\in \Om$. For  all conformal maps 
 $f\: \Om \ra D:=f(\Om)\sub \oC$ with the normalization
\eqref{norm2}  
we have 
\begin{equation}\label{mainineq} 
  \mathcal{S}(f) =2\pi \Re(a_1)+
 \sum_{j=1}^n (V_j^2-A_j)\ge 0 
 \end{equation} 
 with quality if and only if $f$ is the identity
  on $\Om$.  
 
  Moreover, the 
 identity on $\Om$ is the only conformal map $f$  of $\Om$ with the normalization \eqref{norm2}   onto a square domain.     
\end{proposition} 
 
\no {\em Proof.} Before we show the main inequality 
\eqref{mainineq}, let us convince ourselves how the 
last uniqueness statement can be derived from the uniqueness statement in \eqref{mainineq}. 

Indeed, suppose $f\: \Om \ra \Om'$ is a conformal map onto another  square domain domain $\Om'$.
The map $f^{-1}$ on $\Om'$  is also normalized as in
 \eqref{norm2}, where 
 the coefficients $a_1'$ of $f^{-1}$ and $a_1$ of $f$ are related by the equation $a'_1=-a_1$. 

Since the roles of $\Om$ and $\Om'$ are symmetric, 
we may assume that $\Re(a_1)\le 0$; for otherwise, we consider $f^{-1}$ instead of $f$. 
 
 Since $\Om'=f(\Om)$ is a square domain, we have 
 $V_j^2=A_j$ for $j=1, \dots, n$. Hence 
 $\mathcal{S}(f)=\Re(a_1)\le 0$, and so $\mathcal{S}(f)=0$ by \eqref{mainineq}.
 Now from the uniqueness statement in  \eqref{mainineq} we deduce that $f$ is the identity on 
 $\Om$ as desired. 
 
To  prove \eqref{mainineq}, let $f$ be as in the statement.  We consider the rectangle 
$\mathcal{R}=[-l,l]\times [-r,r]\subset \R^2\cong \C$
for large $r>0$. Here we chose $l=r^{2/3}$ 
 so 
that 
\begin{equation}\label{l}
 l/r\to 0 \text{ and } r/l^2 \to 0
 \end{equation} 
 as $r\to \infty$. 
 
 In the following, we assume that $r$ is so large that 
 $\oC\setminus \Om$ is contained in the interior of 
 $\mathcal{R}$.  Then $\partial \mathcal{R}\sub \Om$ and $J=f(\partial R)$
 is a Jordan curve in $\C$.  We want to estimate the area $A=A(r)$ of the region enclosed by the positively oriented Jordan curve  
 $J=f(\partial \mathcal{R})$ up to a term $o(1)$ as $r \to \infty$. 
 
By  the second relation in \eqref{l} we have
 \begin{eqnarray*} 
 A &= & \frac 1{2i} \int_J \overline w\, dw 
 = \frac 1{2i} \int_{\partial \mathcal{R}} 
 \overline {f(z)} f'(z)\, dz\\
 &=& \frac 1{2i}  \int_{\partial \mathcal{R}}
 \overline {\left(z+ \frac{a_1}{z}+\dots \right) }
\left(1- \frac{a_1}{ z^2}+\dots \right)\, dz \nonumber 
 \\
&=& \frac 1{2i}   \int_{\partial \mathcal{R}}
\left(\overline z + \frac {\overline a_1}{\overline z}
- \frac{a_1 \overline{z}}{z^2} + O\left(\frac{1}{|z|^2}\right)\right)\, dz\nonumber  \\
&=& 4rl +   \int_{\partial \mathcal{R}}
\Im\left(\frac {\overline a_1z}{\overline z}\right)
\, \frac {dz}{z}+o(1).  \nonumber
\end{eqnarray*}  
It remains to estimate the last  integral in this equation. Note that for this we may ignore its imaginary part, because the expressions $A$ and 
$4rl$ are real. 

We set $a_1=\alpha+i\beta$ with $\alpha, \beta\in \R$ and $\gamma(t)=l+it$ for $t\in [-r,r]$. Then 
by  the first relation in \eqref{l}
we have  \begin{eqnarray}  \int_{\partial \mathcal{R}}
\Im\left(\frac {\overline a_1z}{\overline z}\right)
\, \frac {dz}{z} &= & 2\int_{\ga}
\Im\left(\frac {\overline a_1z}{\overline z}\right)
\, \frac {dz}{z} +o(1) \nonumber \\
&= & 2\int_{-r}^r
\Im\left(\frac {\overline a_1\ga(t)}{\overline {\ga(t)} }\right)\, \frac{i(l-it)}{l^2+t^2}\, dt + o(1)  \nonumber \\
&= & 2\int_{-r}^r  \frac {2\alpha l t- \beta (l^2-t^2)}{l^2+t^2}\cdot \frac{t}{l^2+t^2}\, dt + o(1) 
\nonumber \\
&= & 4 \alpha \int_{-r}^r \frac { l t^2}{(l^2+t^2)^2}\, dt + o(1)  
 \nonumber \\
 &=& 4 \alpha \int_{-r/l}^{r/l} \frac { u^2}{(1+u^2)^2}\, du + o(1) \nonumber \\
&=&4 \alpha   \int_{-\infty}^{+\infty} \frac { u^2}{(1+u^2)^2}\, du + o(1) = 2\pi \Re(a_1) +o(1).
\nonumber 
 \end{eqnarray}  
Hence 
\begin{equation}\label{Aasymp}
A= 4lr + 2 \pi \Re(a_1) +o(1)
\end{equation} 
as $r\to \infty$.

For $j=1, \dots, n$ we denote by $S_j$ the complementary components 
of $\Om$ and, as before,  by $K_j$ the complementary components of $D=f(\Om)$. We may assume that the labels are chosen such that these components correspond to each other in the sense that if 
$ z\in \Om \to \partial S_j$, then $f(z)\to \partial K_j$ for $j=1, \dots, n$. 

The components $S_j$ are squares (or possibly points as  degenerate squares) 
with sides parallel to the real or  the imaginary axis. Let  $l_j\in [0, \infty)$ be the side length of 
$S_j$.  We now  define a Borel function $\rho\: \C\ra [0, \infty]$ as follows: 
\begin{equation}\label{rho}\rho(z) = \left\{ \begin{array} {cl} |f'(z)| &\text{ for }
z\in \Om,  \\
\displaystyle  V_j/l_j& \text{ for } z \in S_j  \text{ and } l_j>0, \\
\infty & \text{ for } z \in S_j   \text{ and } l_j=0. 
\end{array} \right.
\end{equation}
The rectangle  $\mathcal{R}$ contains all the squares $S_j$ in its interior.
 So by \eqref{Aasymp} we have 
 \begin{eqnarray} \label{intr2}
 \int_{\mathcal{R}} \rho^2 &=& \int_{\mathcal{R}\cap\Om}
 |f'|^2 + \sum_{j=1}^n V_j^2\\
 &=& \text{Area} (f(\mathcal{R}\cap\Om))+ 
 \sum_{j=1}^n V_j^2\nonumber \\
 &=& \text{Area} (f(\mathcal{R}\cap\Om))+
 \sum_{j=1}^n  A_j+ \sum_{j=1}^n (V_j^2-A_j)\nonumber\\
 &=&A+\sum_{j=1}^n (V_j^2-A_j)\nonumber  \\
 &=& 4lr + \mathcal{S}(f) + o(1)  \nonumber
 \end{eqnarray}
as $r\to \infty$. Here and in similar integrals below, integration is with respect to Lebesgue measure, and $\text{Area}(M)$ denotes the  Lebesgue measure of a Borel set $M\sub \C$. 

 For $x\in [-l,l]$  let $\ell_x$be the line segment $\ell_x=\{ x+it: t\in [-r,r]\}$. 
Then we have 
\begin{eqnarray}\label{lint} \int_{\ell_x}\rho(z)\, |dz| &=& \int_{\ell_x\cap \Om} |f'(z)|\, |dz|+
\sum_{S_j\cap \ell_x\ne \emptyset} V_j\\ 
&\ge & \Im (f(x+ir)-f(x-ir)) = 2r + O\left(\frac{1}r\right).\nonumber 
\end{eqnarray}  
Here we used the fact that the union of the set $f(\ell_x\cap \Om)$ together with all sets  $K_j$
with  $S_j\cap \ell_x\ne \emptyset$ forms a connected set joining the points $f(x+ir)$ and $f(x-ir)$. 
 
It follows that 
$$\left(\int_{\ell_x}\rho(z)\, |dz| \right)^2\ge 4r^2+O(1),$$
and using \eqref{l} and \eqref{intr2}
 we conclude 
\begin{eqnarray} \label{lrest} 4lr +o(1)  &\le &\frac 1{2r}\int_{-l}^l\left(\int_{\ell_x}\rho(z)\, |dz|\right)^2\, dx
 \\ 
&\le & \int_{\mathcal{R}}\rho^2 =4lr +\mathcal{S}(f)+o(1). \nonumber
\end{eqnarray}  
This implies $\mathcal{S}(f)\ge o(1)$ as $r\to \infty$, and so $\mathcal{S}(f)\ge 0$ as desired.  

Suppose we have the equality  $\mathcal{S}(f)=0$. Then 
 \begin{equation}\label{aseq}
 \int_{\mathcal{R}}\rho^2 =4lr +o(1)
 \end{equation}
 as $r\to \infty$. 
 If we define $\widetilde \rho=(1+\rho)$, then 
 by \eqref{lint} we have 
 $$   \int_{\ell_x}\widetilde\rho(z)\, |dz|\ge 4r +O\left(\frac{1}r\right), $$
 and so 
 \begin{eqnarray*}  16lr +o(1)  &\le &\frac 1{2r}\int_{-l}^l\left(\int_{l_x}\widetilde
 \rho(z)\, |dz|\right)^2\, dx
 \\ 
&\le & \int_{\mathcal{R}}\widetilde \rho^2.
\end{eqnarray*} 
Hence by \eqref{aseq},
$$ \int_{\mathcal{R}} (1-\rho)^2 = \int_{\mathcal{R}} (2+2\rho^2-\widetilde\rho^2)
\le 8lr+8lr-16lr+o(1)=o(1). 
$$ 
Letting  $r\to\infty$, we conclude that $\int_\C(1-\rho)^2=0$ and so $\rho=1$ almost 
everywhere on $\C$. In particular, $|f'(z)|=1$ for all $z\in \Om$, and so $f'$ is constant. With the given normalization this implies  $f'\equiv 1$ and so $f$ is the identity on $\Om$. 
\hfill $\Box$ 

The proof of Theorem~\ref{main} is complete. 

\section{Remarks}
\label{s:rem} 

1.\ It is not clear which conformal map should maximize the functional in \eqref{funct}. 

2.\ It would be very interesting to see whether the functional   \eqref{funct} can be used to give an independent existence proof for a conformal map of a finitely connected domain $\Om$ onto a square domain $D$ without resorting to the Brandt-Harrington Uniformization Theorem. By a normal family argument one can show the existence of a minimizer, but there seems to be  no simple variational argument in the class of {\em conformal maps} that shows that the minimizer is a conformal map onto a square domain. One can formulate a more general variational problem for real-valued functions $v$ (corresponding to the imaginary parts of the conformal maps) whose solution should give the existence of a suitable conformal map. This is in the same spirit  as classical potential-theoretic methods for the solution of uniformization problems (see \cite{Cou}). 

3.\ Behind the proof of Proposition~\ref{prop:main} is essentially an
asymptotic estimate for the conformal modulus (or extremal length) of a path family 
(see \cite[Chapter 4]{Ahl}), namely the family of  line segments
parallel to the imaginary axis and joining the top to the bottom side of the rectangle $\mathcal{R}$. 
The use of the test function in \eqref{rho} was inspired by Schramm's notion  of  transboundary extremal length  as introduced in  \cite{S95}. 

4.\ Using  asymptotic estimates for modulus in conformal mapping theory is a fairly standard idea. For example, the notion of reduced extremal length as discussed in \cite[Section 4.14]{Ahl} is closely related to this. 
The idea of using a rectangle with side lengths satisfying the conditions in \eqref{l} seems to be new. 

5.\ Very similar arguments as  in this note can be used to show the following (essentially known) fact. 

\begin{proposition}  Let $\Om$ be  a finitely connected  parallel slit domain with slits parallel to the imaginary axis. Then  
$\Re(a_1)\ge 0$ for all conformal map $f\: \Om \ra f(\Om)\sub \oC$ normalized as in \eqref{norm} with equality if and only if $f$ is the identity on $\Om$.
\end{proposition}

\begin{proof} We use the notation as in the proof of Proposition~\ref{prop:main}, but set   
$\rho =|f'|$ on $\Om$ and $\rho=0$ elsewhere. Then 
$$ \int_{\mathcal{R}}\rho^2= \text{Area}(f(\mathcal{R}\cap\Om))\le A=4lr+2\pi \Re(a_1)+o(1)$$ as $r\to \infty$.  We have $\ell_x\sub \Om$ for  every $x\in [-r,r]$ with at most finitely many exceptions. For these values $x$ we have  an estimate as in \eqref{lint}  and obtain a lower bound for $ \int_{\mathcal{R}}\rho^2$ as in 
\eqref{lrest}.  The desired inequality $\Re(a_1)\ge 0$   follows. In case of equality we have 
$$ \int_{\mathcal{R}}\rho^2\le 4lr+o(1).$$ As in the last part of the proof of Proposition~\ref{prop:main}, this  leads to $\rho=1$ almost everywhere on $\C$, and so $f$ must be the identity on $\Om$.  
\end{proof}


\begin{thebibliography}{BHK}

\bibitem[Ahl]{Ahl} L.V.~Ahlfors, \emph{Conformal Invariants}, AMS Chelsea Publishing, Providence, RI, 2010.

\bibitem[Cou]{Cou} R.~Courant {\em Dirichlet's Principle, Conformal Mapping, and Minimal Surfaces}, 
  Springer,  New York-Heidelberg, 1977.
  
  

\bibitem[HS]{HS} Z.X.~He, O.~Schramm, {\em Fixed points, Koebe uniformization and circle packings}, Ann. of Math. (2) 137 (1993), 369--406. 

\bibitem[Go]{Go} G.M.~Goluzin, \emph{Geometric Theory of Functions of a Complex Variable},
Translations of Math.\ Monographs, Vol.~25, Amer.\ Math. Soc., Providence, RI, 1969. 

\bibitem[S96] {S96} O. Schramm, \emph{Conformal uniformization and
packings,} Israel J. Math., 93 (1996), 399--428.



\bibitem[S95]{S95} O. Schramm, \emph{Transboundary extremal length,}
J. d'Analyse Math.~66 (1995), 307--329.

\end{thebibliography}
\end{document}